\documentclass[12pt,a4paper,reqno,twoside]{amsart}

\usepackage[english]{babel}
\usepackage{stmaryrd}
\usepackage{dsfont}
\usepackage[symbol*,ragged]{footmisc}
\usepackage[colorlinks,linkcolor=red,anchorcolor=blue,citecolor=blue,urlcolor=blue]{hyperref}
\usepackage{color,xcolor}

\usepackage{geometry}
\usepackage{amssymb}
\usepackage{amsmath}
\usepackage{mathrsfs}
\usepackage{amsfonts}
\usepackage{epsfig}

\usepackage{amsthm}
\usepackage{amsxtra}
\usepackage{bbding}
\usepackage{epsfig}
\usepackage{graphicx}
\usepackage{latexsym}
\usepackage{mathbbol}
\usepackage{bbold}

\usepackage{pifont}
\usepackage{wasysym}
\usepackage{skull}
\usepackage{float}

\DeclareSymbolFontAlphabet{\mathbb}{AMSb}
\DeclareSymbolFontAlphabet{\mathbbol}{bbold}

\usepackage{amscd}
\usepackage[all]{xy}
\allowdisplaybreaks[4]
\usepackage{setspace}

\theoremstyle{plain}
\newtheorem{theorem}{\normalfont\scshape Theorem}[section]
\newtheorem{corollary}{\normalfont\scshape Corollary}[section]
\newtheorem{proposition}{\normalfont\scshape Proposition}[section]
\newtheorem{lemma}[proposition]{\normalfont\scshape Lemma}
\newtheorem*{corollary*}{\normalfont\scshape Corollary}

\theoremstyle{remark}
\newtheorem*{remark*}{\normalfont\scshape Remark}

\numberwithin{equation}{section}
\addtocounter{footnote}{1}

\newcommand{\Rmnum}[1]{\uppercase\expandafter{\romannumeral #1}}  

\renewcommand{\footnoterule}{
	\kern -3pt
	\hrule width 2.5in height 0.4pt
	\kern 3pt
}

\makeatletter
\@ifundefined{MakeUppercase}{}{}
\makeatother

\begin{document}
\title[On Higher--Power Moments of $ \Delta_a(x) $ for $-1/2<a<0$]
	  {On Higher--Power Moments of $ \Delta_a(x) $ for $-1/2<a<0$}
	
\author[Yi Cai, Jinjiang Li, Yankun Sui, Fei Xue, Min Zhang]
{Yi Cai \,\, \& \,\, Jinjiang Li \,\, \& \,\, Yankun Sui \,\, \& \,\, Fei Xue \,\, \& \,\, Min Zhang}
	
\address{Department of Mathematics, China University of Mining and Technology,
		 Beijing 100083, People's Republic of China}
	
\email{yi.cai.math@gmail.com}

\address{Department of Mathematics, China University of Mining and Technology,
		Beijing 100083, People's Republic of China}
	
\email{jinjiang.li.math@gmail.com}

\address{(Corresponding author) College of Information Engineering, Nanjing Xiaozhuang University,
        Nanjing 211171, Jiangsu, People's Republic of China}
	
\email{syk\_cumtb@163.com}

\address{Department of Mathematics, China University of Mining and Technology,
		Beijing 100083, People's Republic of China}
	
\email{fei.xue.math@gmail.com}

\address{School of Applied Science, Beijing Information Science and Technology University,
		 Beijing 100192, People's Republic of China  }

\email{min.zhang.math@gmail.com}

\date{}
	
\footnotetext[1]{Yankun Sui is the corresponding author. 
		\quad\,\,
{\textbf{MR(2020) Subject Classification}}: 11N37, 11M06     }

\begin{abstract}
Let $-1/2<a<0$ be a fixed real number and	
\begin{equation*}
\Delta_{a}(x)=\sideset{}{'}\sum_{n\leq x}
\sigma_a(n)-\zeta(1-a)x-\frac{\zeta(1+a)}{1+a}x^{1+a}+\frac{1}{2}\zeta(-a).
\end{equation*}
In this paper, we investigate the higher--power moments of $\Delta_a(x)$ and give the corresponding asymptotic formula for the integral $\int_{1}^{T}\Delta_a^k(x)\mathrm{d}x$, which constitutes an improvement upon the previous result of Zhai \cite{Zhai200402} for $k=3,4,5$ and an enlargement of the upper bound of $k$ to $7$.
\end{abstract}

\maketitle

\section{Introduction and Main Results}\label{intro}
Let $-1/2<a<0$ be a fixed real number and $\sigma_a(n)$ denote the sum of $a$--th powers of positive divisors of $n$, i.e., $\sigma_a(n)=\sum_{d|n}d^a$.  Define
\begin{equation}
\Delta_{a}(x):=\sideset{}{'}\sum_{n\leq x}
\sigma_a(n)-\zeta(1-a)x-\frac{\zeta(1+a)}{1+a}x^{1+a}+\frac{1}{2}\zeta(-a),
\end{equation}
where $\zeta(\cdot)$ is Riemann zeta--function, and $\sideset{}{'}{\scriptstyle\sum}\limits_{n\leq x}$ means that the final term should be weighted with 1/2 if $x$ is a positive integer. We recall the classical error term of Dirichlet's divisor problem, i.e.,
\begin{equation*}
\Delta(x)=\sideset{}{'}\sum\limits_{n\leq x}d(n)-x\log x-(2\gamma-1)x,
\end{equation*}
where $\gamma$ is Euler's constant, and $d(n)$ denotes the Dirichlet divisor function. By comparing the definitions of
$\Delta(x)$ and $\Delta_a(x)$, one can easily see that $\sigma_a(n)\to d(n)$ when $a\to 0$, and
\begin{equation*}
\lim_{a\to 0}\Delta_{a}(x)=\Delta(x).
\end{equation*}
Many mathematicians have investigated the properties of $\Delta(x)$, and numerous results, which contain the upper bound estimates and higher power moments of $\Delta(x)$, have been established.  However, only a few results are known about $\Delta_{a}(x)$.  In this paper, we concentrate on the higher power moments of $\Delta_{a}(x)$ for the case $-1/2<a<0$.
	
By an elementary argument from the definition of $\Delta_{a}(x)$, one can deduce that $\Delta_{a}(x)\ll x^{(1+a)/2}$. Later, in 1987, the exponent $(1+a)/2$ was improved to $1/(3-2a)$ by Kiuchi \cite{kiuchi1987}. After that, in 1988, P\'etermann \cite{petermann-1988} sharpened the result of Kiuchi \cite{kiuchi1987} by representing the upper bound estimate in terms of exponent pairs. Suffice it to say that it implies at least that  $\Delta_{a}(x)\ll x^{(1+a)/3+\varepsilon}$ for any positive $\varepsilon$, which proved by Vinogradov  \cite{vinogradov1985}. However, the precise order of expected magnitude of $\Delta_{a}(x)$ is as yet unknown.
	
In 1987, Kiuchi \cite{kiuchi1987} considered the situation in which $\sigma_a(n)$ is multiplied by $e^{2\pi inh/k}$, where $h$ and $k$ are coprime integers. In the case $k=1$, Kiuchi \cite{kiuchi1987} proved that the asymptotic formula
\begin{equation}\label{1.4}
	\int_{T}^{2T}\Delta_{a}^2(x)\mathrm{d}x=C_2(a)T^{3/2+a}+O(T^{5/4+a/2+\varepsilon})
\end{equation}
holds for $-1/2<a<0$. In 1996, Meurman \cite{meurman1996} considered the truncated Vorono\"{\i}'s formula of $\Delta_a(x)$, which states that
\begin{equation}
\Delta_{a}(x)=\frac{x^{1/4+a/2}}{\sqrt{2}\pi}\sum_{n\leq N}\frac{\sigma_a(n)}{n^{3/4+a/2}}
\cos\left(4\pi\sqrt{nx}-\frac{\pi}{4}\right)+O\left(x^{1/2+\varepsilon}N^{-1/2}\right)
\end{equation}
holds for $-1/2<a<0$ and $1\leq N\ll x$. Furthermore, Meurman \cite{meurman1996} also reestablished the quadratic
power moment of $\Delta_a(x)$ and gave stronger error term with larger power saving in (\ref{1.4}), i.e.,
\begin{equation}\label{1.4_2}
	\int_{T}^{2T}\Delta_{a}^2(x)\mathrm{d}x=C_2(a)T^{3/2+a}+O(T), \qquad(-1/2<a<0) .
\end{equation}
If one follows the lines of Ivi\'c's technique in \cite{ivic1985}, then we will obtain the following $ L^{A_0} $--means estimate
\begin{equation}\label{1.6}
	\int_{1}^{T}|\Delta_{a}(x) |^{A_0}\mathrm{d}x\ll T^{(4+A_0+2A_0a)/4+\varepsilon}
\end{equation}
with  little effort, where $A_0$ is a fixed real numbers depending on $a$. Based on (\ref{1.4})--(\ref{1.6}), it is reasonable to conjecture that $ \Delta_a(x) \ll x^{1/4+a/2+\varepsilon}$. However, this conjecture is perhaps out of reach at present.
	
For higher power moments of $\Delta_{a}(x)$, Zhai \cite{Zhai200401}, in 2004, investigated the third, fourth and fifth power moments of $\Delta_{a}(x)$, who proved that
\begin{align*}
\int_{T}^{2T}\Delta_{a}^3(x)\mathrm{d}x=C_3(a)T^{\frac{7+6a}{4}}
		+O\Big(T^{\frac{7+6a}{4}-\frac{3+8a-12a^2}{12-24a}+\varepsilon}\Big), \qquad \bigg(-\frac{\sqrt{13}-2}{6}<a<0\bigg),
\end{align*}
\begin{align*}
\int_{T}^{2T}\Delta_{a}^4(x)\mathrm{d}x = C_4(a)T^{2+2a}\nonumber
+O\Big(T^{2+2a-\min\big(  \frac{1+6a-8a^2}{8-16a},\frac{1}{7}(\frac{1}{3}+2a)  \big )+\varepsilon}\Big),  \quad\bigg(-\frac{\sqrt{17}-3}{8}<a<0\bigg),
\end{align*}
and
\begin{align*}
\int_{T}^{2T}\Delta_{a}^5(x)\mathrm{d}x=C_5(a)T^{\frac{9+10a}{4}}
+O\Big(T^{\frac{9+10a}{4}-\frac{1+30a}{180}+\varepsilon}\Big), \quad \left(-\frac{1}{30}<a<0\right)	.
	\end{align*}
However, the arguments of Zhai \cite{Zhai200401} fail for $k\geq 6$ .
	
For convenience, we give some definitions and notations as follows. Suppose that $k\geq 2$ is a fixed integer. Define
\begin{equation}\label{1.7}		
s_{k;l}(\sigma_a):=\sum_{\substack{\sqrt{n_1}+\cdots+\sqrt{n_l}=\sqrt{n_{l+1}}+\cdots+\sqrt{n_k}\\ n_1,\dots,n_l,n_{l+1},\dots,n_k\in\mathbb{N}^*}}
\frac{\sigma_a(n_1)\cdots\sigma_a(n_k)}{(n_1\cdots n_k)^{3/4+a/2}},
\quad(1\leq l<k),
\end{equation}
\begin{equation}\label{1.8}
s_{k;l}(\sigma_a;y):=\sum\limits_{\substack{\sqrt{n_1}+\cdots+\sqrt{n_l}
=\sqrt{n_{l+1}}+\cdots+\sqrt{n_k}\\n_1,\dots,n_l,n_{l+1},\dots,n_k\leq y}}
\frac{\sigma_a(n_1)\cdots \sigma_a(n_k)}{(n_1\cdots n_k)^{3/4+a/2}}, \quad(1\leq l<k),
\end{equation}
\begin{equation}\label{1.9}
B_k(\sigma_a):=\sum_{l=1}^{k-1}{k-1 \choose l }s_{k;l}(\sigma_a)\cos \frac{\pi(k-2l)}{4},\qquad
C_k(a):=\frac{B_k(\sigma_a)}{(\sqrt{2}\pi)^k2^{k-1}}.	
\end{equation}
Let $A_0>2$ be a real number, and define
\begin{equation}\label{1.10}
b_a(k):=2^{k-2}+(k-6)/4-ka/2,
\end{equation}
\begin{equation}\label{1.11}
\alpha(k,A_0):=
\begin{cases}
1/4+a/2,  & \textrm{if}\,\, k-1<A_0/2, \\
\frac{A_0-A}
{2(A_0-2)}+\frac{a(A_0-A)}{A_0-A},  & \textrm{if}\,\, A_0/2+1\leq k<A_0,
\end{cases}
\end{equation}
\begin{equation}\label{1.12}
\delta_a(k,A_0):=\frac{\alpha(k,A_0)}{(2b_a(k)+2\alpha(k,A_0))}.
\end{equation}	
For given functions $F$ and $G$, the notations $F \ll G$, $G \gg F$ and $F = O(G)$ are all equivalent to the statement that the inequality $|F| \leq C|G|$ holds with some constant $C > 0$.  Furthermore, $\varepsilon$ always denotes a sufficiently small positive constant, which may not be the same at different occurrences.  Above all, in this paper, we shall continue to improve the result of Zhai \cite{Zhai200401}, and establish the
following theorem and corollaries.
\begin{theorem}\label{maintheorem}
Suppose that $A_0>2$ is a real number such that
\begin{equation*}
 \int_{1}^{T}|\Delta_{a}(x) |^{A_0}\mathrm{d}x\ll T^{(4+A_0+2A_0a)/4+\varepsilon}.
\end{equation*}
Then, for any integer $3\leq k<A_0$, one has
\begin{align}
\int_{1}^{T}\Delta_{a}^{k}(x)\mathrm{d}x =C_k(a)
T^{(4+k+2ka)/4}+O\big(T^{(4+k+2ka)/4-\delta_a(k,A_0)+\varepsilon}\big),
\end{align}
where $C_k(a)$ and $\delta_a(k,A_0)$ are defined in (\ref{1.9}) and (\ref{1.12}), respectively.
Moreover, we can take $A_0=8(1-a^2)/(1-2a)$, and derive the $k$--th power moment of $\Delta_{a}(x)$ for
$3\leqslant k\leqslant7$.
\end{theorem}
\begin{remark*}
For $k\geq 8$, Theorem \ref{maintheorem} is only a conditional result. If we have a stronger result of the upper bound estimate of $\Delta_a(x)$, we can enlarge the range of $A_0$ for $-1/2<a<0$. Then, the asymptotic formula holds for a wider range of $k$.
\end{remark*}

\begin{corollary}
Suppose that $-1/2<a<0$. Then one has
\begin{equation}
\int_{1}^{T}\Delta_{a}^{3}(x)\mathrm{d}x=C_3(a)
T^{7/4+3a/2}+O\big(T^{7/4+3a/2-\delta_a(3,A_0)+\varepsilon}\big),
\end{equation}
where $A_0=8(1-a^2)/(1-2a)>3$. Moreover, there holds
\begin{equation*}
\delta_a(3,A_0) =
\begin{cases}
\displaystyle\frac{(1+2a)^2(-5+4a)}{8(5+2a-7a^2+2a^3)} >0,
&\text{if}\,\displaystyle-1/2<a\leq\displaystyle-\frac{\sqrt{3}-1}{2}, \smallskip \\
\smallskip
\displaystyle\frac{1+2a }{12-8a }>0, & \text{if}\,\displaystyle-\frac{\sqrt{3}-1}{2}<a<0 \smallskip .
\end{cases}
\end{equation*}
\end{corollary}

\begin{corollary}
Suppose that $-(\sqrt{3}-1)/2<a<0$. Then one has
\begin{equation}
\int_{1}^{T}\Delta_{a}^{4}(x)\mathrm{d}x=C_4(a)
T^{2+2a}+O\big(T^{2+2a-\delta_a(4,A_0)+\varepsilon}\big),
\end{equation}
where $A_0=8(1-a^2)/(1-2a)>4$. Moreover, there holds
\begin{equation*}
\delta_a(4,A_0) =
\begin{cases}
\displaystyle\frac{1+4a+2a^2-4a^3}{23+10a-32a^2+8a^3} >0,
   & \text{if}\,\displaystyle-\frac{\sqrt{3}-1}{2}<a\leq\displaystyle-\frac{\sqrt{13}-3}{4}, \smallskip \\
\smallskip
\displaystyle\frac{1+2a }{30-12a }>0, & \text{if}\,\displaystyle-\frac{\sqrt{13}-3}{4}<a<0 \smallskip .
\end{cases}
\end{equation*}
\end{corollary}

\begin{corollary}
Suppose that $-1/4<a<0$. Then one has
\begin{equation}
\int_{1}^{T}\Delta_{a}^{5}(x)\mathrm{d}x =C_5(a) T^{9/4+5a/2}+O\left(T^{9/4+5a/2-\delta_a(5,A_0)+\varepsilon}\right),
\end{equation}
where
\begin{equation*}
A_0=8(1-a^2)/(1-2a)>5,
\end{equation*}
and
\begin{equation*}
\delta_a(5,A_0) = \frac{(1+2a)(-3-10a+8a^2) }{24(8+4a-11a^2+2a^3)} >0.
\end{equation*}
\end{corollary}
	
\begin{corollary}
Suppose that $-(\sqrt{13}-3)/4<a<0$. Then one has
\begin{equation}
\int_{1}^{T}\Delta_{a}^{6}(x)\mathrm{d}x =C_6(a) T^{5/2+3a}+O(T^{5/2+3a-\delta_a(6,A_0)+\varepsilon}),
\end{equation}
where
\begin{equation*}
A_0=8(1-a^2)/(1-2a)>6,
\end{equation*}
and
\begin{equation*}
\delta_a(6,A_0) = \frac{1+8a+8a^2-8a^3 }{194+108a-264a^2+32a^3} >0.
\end{equation*}
\end{corollary}
	
\begin{corollary}
Suppose that $-(\sqrt{57}-7)/8<a<0$. Then one has
\begin{equation}
\int_{1}^{T}\Delta_{a}^{7}(x)\mathrm{d}x =C_7(a) T^{11/4+7a/2}+O(T^{11/4+7a/2-\delta_a(7,A_0)+\varepsilon}),
\end{equation}
where
\begin{equation*}
A_0=8(1-a^2)/(1-2a)>7,
\end{equation*}
and
\begin{equation*}
\delta_a(7,A_0) = \frac{1+16a+20a^2-16a^3 }{776+464a-1048a^2+80a^3} >0.
\end{equation*}
\end{corollary}

\section{Preliminary Lemmas}
In this section, we shall demonstrate some lemmas which will be used to establish Theorem \ref{maintheorem}.

\begin{lemma}\label{lemma-inte}
If $g(x)$ and $h(x)$ are continuous real--valued functions of $x$ and $g(x)$ is monotonic, then
\begin{equation*}
\int_a^bg(x)h(x)\mathrm{d}x\ll\bigg(\max_{a\leqslant x\leqslant b}\big|g(x)\big|\bigg)
\bigg(\max_{a\leqslant u<v\leqslant b}\bigg|\int_u^vh(x)\mathrm{d}x\bigg|\bigg).
\end{equation*}
\end{lemma}
\begin{proof}
See Lemma 1 of Tsang \cite{Tsang1992}.
\end{proof}

\begin{lemma}\label{lemma2.1}
Suppose that $A, B\in \mathbb{R}$ and $ A\neq 0$. Then one has
\begin{equation*}
 \int_{T}^{2T}\cos(A\sqrt{t}+B)\mathrm{d}t\ll T^{1/2}|A|^{-1}.
\end{equation*}
\end{lemma}
\begin{proof}
It follows from Lemma \ref{lemma-inte} easily.
\end{proof}

\begin{lemma}\label{lemma2.3}
Suppose that $k\geq 3$ and $(i_1,\dots,i_{k-1})\in \{0,1\}^{k-1} $ such that
\begin{equation*}
\sqrt{n_1}+(-1)^{i_1}\sqrt{n_2}+(-1)^{i_2}\sqrt{n_3}+\dots+(-1)^{i_{k-1}}\sqrt{n_k}\neq0.
\end{equation*}
Then one has
\begin{equation*}
|\sqrt{n_1}+(-1)^{i_1}\sqrt{n_2}+(-1)^{i_2}\sqrt{n_3}+\cdots+(-1)^{i_{k-1}}\sqrt{n_k}|\gg \max(i_1,\dots,i_{k-1})^{-(2^{k-2}-2^{-1})}.
\end{equation*}
\end{lemma}
\begin{proof}
The cases $k = 3, 4$ are Lemma 2 and Lemma 3 of Tsang \cite{Tsang1992}, respectively. For the general case,
one can refer to Lemma 2.2 of Zhai \cite{Zhai200402}.
\end{proof}

\begin{lemma}\label{lemma2.4}
Suppose that $k\geq3$, $(i_1,\dots,i_{k-1})\in \{0,1\}^{k-1} $ with $(i_1,\dots,i_{k-1})\neq(0,\dots,0)$, $N_1,\dots,N_k>1$, $0<\Delta\ll E^{1/2}$, $E=\max(N_1,\dots,N_k)$. Let
\begin{equation*}
	\mathcal{A}:=\mathcal{A}(N_1,\dots,N_k;i_1,\dots,i_{k-1};\Delta)
\end{equation*}
denote the number of solutions of the inequality
\begin{equation*}
\left|\sqrt{n_1}+(-1)^{i_1}\sqrt{n_2}+(-1)^{i_2}\sqrt{n_3}+\cdots+(-1)^{i_{k-1}}\sqrt{n_k}\right|<\Delta
\end{equation*}
with $N_j<n_j\leq 2N_j\,(1\leq j\leq k)$. Then, there holds
\begin{equation*}
 \mathcal{A}\ll\Delta E^{-1/2}N_1\cdots N_k+E^{-1}N_1\cdots N_k.
\end{equation*}
\end{lemma}
\begin{proof}
See Lemma 2.4 of Zhai \cite{Zhai200402}.
\end{proof}

\begin{lemma}\label{lemma2.5}
Suppose that $y>1$ is a large parameter. Let $s_{k;l}(\sigma_a)$ and $s_{k;l}(\sigma_a;y)$ be defined as in (\ref{1.7}) and (\ref{1.8}), respectively. Then there holds
\begin{equation*}
\left|s_{k;l}(\sigma_a)-s_{k;l}(\sigma_a;y)\right|\ll y^{-1/2-a+\varepsilon}.
\end{equation*}
\end{lemma}
\begin{proof}
By following the lines exactly the same as that of Lemma 3.1 in \cite{Zhai200402}, one can easily deduce the conclusion. Thus, we omit the details herein.
\end{proof}

\begin{lemma}\label{lemma2.6}
Suppose that $\Delta_{a}(x) \ll x^\theta$ holds for $\theta>(1+2a)/4 $. Then one has
\begin{equation}
	\int_{1}^{T}|\Delta_{a}(x) |^A\mathrm{d}x\ll T^{(4+A+2Aa)/4+\varepsilon},
\end{equation}
where $ 0 \leq A\leq \frac{2\theta(1-a)}{\theta-(1+2a)/4} $.
\end{lemma}
\begin{proof}
The conclusion of this lemma can be established by following the technique of large--value estimate in Ivi\'c \cite{ivic1983}. Let $T\leq t_1<\cdots <t_k\leq 2T$, and $|t_r-t_s|\geq V$ for $r\neq s \leq R$. If $\left| \Delta_a(t_r)\right|\geq V\gg T^{7/4(8-9a)+\varepsilon}$ for $r\leq R$ , one has
\begin{equation*}
	R\ll T^{\varepsilon}(TV^{-3+2a}+T^{15/4}V^{-12+11a}).
\end{equation*}
Then, it is sufficient to establish the desired upper bound estimate over $[T/2,T]$. As is shown in the process
of Theorem 2 in \cite{ivic1983}, one can deduce the result immediately by following the lines exactly the same as that. Hence, we omit the details herein.
\end{proof}

\section{Proof of Theorem \ref{maintheorem}}
In this section, we focus on establishing Theorem \ref{maintheorem}.
From Lemma \ref{lemma2.6}, we know that the value of $A_0$, which makes (\ref{1.6}) hold, depends on the large--value estimate combined with the upper bound estimate of $\Delta_{a}(x)$. If we insert the estimate $\Delta_a(x)\ll x^{(1+a)/3+\varepsilon}$ into Lemma \ref{lemma2.6}, it is easy to see that (\ref{1.6}) holds with $A_0=8(1-a^2)/(1-2a)$. Under such conditions, for $3\leq k\leq7$, we can obtain the asymptotic formula of $k$--th power moments of $\Delta_a(x)$ with distinct ranges of $a$.

Suppose that $T\geq 10$ is a real number.
By splitting arguments, it suffices to evaluate the integral $\int_{T}^{2T}\Delta_{a}^{k}(x)\mathrm{d}x$. Suppose that $y$ is a parameter which satisfies $T^{\varepsilon} <y\leq T$. For any $T\leq x\leq 2T$, we define
\begin{align*}
&\mathcal{R}_{a1}=\mathcal{R}_{a1}(x,y):=\frac{x^{1/4+a/2}}{\sqrt{2}\pi}\sum_{n\leq y}\frac{\sigma_a(n)}{n^{3/4+a/2}}\cos\bigg(4\pi \sqrt{nx}-\frac{\pi}{4}\bigg),\\
&\mathcal{R}_{a2}=\mathcal{R}_{a2}(x,y):=\Delta_{a}(x)-\mathcal{R}_{a1}.
\end{align*}	
We shall show that the higher--power moments of $\mathcal{R}_{a2}$ is small, and hence the integral $\int_{T}^{2T}\Delta_a^k(x)\mathrm{d}x$ can be well approximated by $\int_{T}^{2T}\mathcal{R}_{a1}^k\mathrm{d}x$. Suppose that $h\geq 3$ is a fixed integer. By the elementary property of cosine function, one has
\begin{equation*}
\cos a_1\cdots\cos a_h=\frac{1}{2^{h-1}}\sum_{(i_1,\dots,i_{h-1})\in\{0,1\}^{h-1}}
\cos\left(a_1+(-1)^{i_1}a_2+(-1)^{i_2}a_3+\cdots+(-1)^{i_{h-1}}a_h\right),
\end{equation*}
which combined with the definition of $\mathcal{R}_{a1}$ yields
\begin{align*}
		 \mathcal{R}_{a1}^h
= & \,\, \frac{x^{h/4+ha/2}}{(\sqrt{2}\pi)^h}\sum_{n_1\leq y}\sum_{n_2\leq y}\cdots\sum_{n_h\leq y}
         \frac{\sigma_a(n_1)\cdots\sigma_a(n_h)}{(n_1\cdots n_h)^{3/4+a/2}}\prod_{j=1}^{h}\cos\left(4\pi \sqrt{n_jx}-\frac{\pi}{4}\right)
                \nonumber \\
= & \,\, \frac{x^{h/4+ha/2}}{(\sqrt{2}\pi)^h2^{h-1}}\sum_{(i_1,\dots,i_{h-1})\in\{0,1\}^{h-1}}
         \sum_{n_1\leq y}\cdots\sum_{n_h\leq y}\frac{\sigma_a(n_1)\cdots\sigma_a(n_h)}{(n_1\cdots n_h)^{3/4+a/2}}
                \nonumber \\
  & \,\, \times\cos\left(4\pi\sqrt{x}\alpha(n_1,\dots ,n_h;i_1,\dots,i_{h-1})-\frac{\pi}{4}
         \beta(i_1,\dots ,i_{h-1})\right) ,
\end{align*}
where 	
\begin{equation*}	
\alpha(n_1,\dots,n_h;i_1,\dots,i_{h-1}):=\sqrt{n_1}+(-1)^{i_1}\sqrt{n_2}+(-1)^{i_2}\sqrt{n_3}
+\cdots+(-1)^{i_{h-1}}\sqrt{n _h},
\end{equation*}
\begin{equation*}
\beta(i_1,\dots,i_{h-1}):=1+(-1)^{i_1}+(-1)^{i_2}+\cdots+(-1)^{i_{h-1}}.
\end{equation*}
Thus, we can write
\begin{equation}\label{3.1}
	\mathcal{R}_{a1}^h=\frac{1}{(\sqrt{2}\pi)^h2^{h-1}}(S_1(x)+S_2(x)),
\end{equation}
where
\begin{align}		&S_1(x)=x^{h/4+ha/2}\sum_{(i_1,\dots,i_{h-1})\in\{0,1\}^{h-1}}\cos\left(-\frac{\pi}{4}\beta\right)
 \sum\limits_{\substack{n_j\leq y,\,\,1\leq j\leq h\\ \alpha=0}}
 \frac{\sigma_a(n_1)\cdots\sigma_a(n_h)}{(n_1\cdots n_h)^{3/4+a/2}},\\
&S_2(x)=x^{h/4+ha/2}\sum_{(i_1,\dots,i_{h-1})\in\{0,1\}^{h-1}}\sum\limits_{\substack{n_j\leq y,\,\,1\leq j\leq h\\ \alpha\neq 0}}\frac{\sigma_a(n_1)\cdots\sigma_a(n_h)}{(n_1\cdots n_h)^{3/4+a/2}}\cos\left(4\pi\alpha\sqrt{x}-\frac{\pi}{4}\beta\right),
\end{align}
\begin{equation*}
\alpha=\alpha(n_1,\dots,n_h;i_1,\dots,i_{h-1}),\quad\beta=\beta(i_1,\dots,i_{h-1}).	
\end{equation*}
First of all, we consider the contribution of $S_1(x)$. We have
\begin{align}
\int_{T}^{2T}S_1(x)\mathrm{d}x=\sum_{(i_1,\dots,i_{h-1})\in \{0,1\}^{h-1}}\cos\left(-\frac{\pi}{4}\beta\right)   \sum\limits_{\substack{n_j\leq y,\,\,1\leq j\leq h\\ \alpha= 0}}\frac{\sigma_a(n_1)\cdots
\sigma_a(n_h)}{(n_1\cdots n_h)^{3/4+a/2}}	\int_{T}^{2T}x^{h/4+ha/2}\mathrm{d}x.
\end{align}
If $\alpha=0$, then $1\in\{i_1,\dots,i_{h-1}\}$. Let $l=i_1+\cdots+i_{h-1}$. Thus, we have
\begin{equation*}
\sum\limits_{\substack{n_j\leq y,\,\,1\leq j\leq h\\ \alpha= 0}}
\frac{\sigma_a(n_1)\cdots\sigma_a(n_h)}{(n_1\cdots n_h)^{3/4+a/2}}=s_{h;l}(\sigma_a;y),
\end{equation*}
where the $s_{h;l}(\sigma_a;y)$ is defined as in (\ref{1.8}). By Lemma \ref{lemma2.5}, one has
\begin{equation}	\int_{T}^{2T}S_1(x)\mathrm{d}x=B_h^*(\sigma_a)\int_{T}^{2T}x^{h/4+ha/2}\mathrm{d}x
+O\left(T^{(4+h+2ha)/4+\varepsilon}y^{-1/2-a}\right),
\end{equation}
where 	
\begin{equation*}
B_h^*(\sigma_a):=\sum_{(i_1,\dots,i_{h-1})\in \{0,1\}^{h-1}}\cos\left(-\frac{\pi\beta}{4}\right)
\sum\limits_{\substack{(n_1,\dots,n_{h})\in\mathbb{N}^{h}\\\alpha=0}}
\frac{\sigma_a(n_1)\cdots\sigma_a(n_h)}{(n_1\cdots n_h)^{3/4+a/2}}.
\end{equation*}
For any $(i_1,\dots,i_{h-1})\in\{0,1\}^{h-1}\setminus\{(0,\dots,0)\}$, let
\begin{equation*}	S(\sigma_a;i_1,\dots,i_{h-1})=\sum\limits_{\substack{(n_1,\dots,n_{h})\in\mathbb{N}^{h}\\\alpha=0}}
\frac{\sigma_a(n_1\cdots\sigma_a(n_h))}{(n_1\cdots n_h)^{3/4+a/2}},
\end{equation*}
and
\begin{equation*}
	l(i_1,\dots,i_{h-1})=i_1+\cdots+i_{h-1}.
\end{equation*}
It is easy to see that if $l(i_1,\dots,i_{h-1})=l(i_1',\cdots,i_{h-1}')$ or  $l(i_1,\dots,i_{h-1})+l(i_1',\dots,i_{h-1}')=h$, then one has
\begin{equation*}
S(\sigma_a;i_1,\dots,i_{h-1})=S(\sigma_a;i_1',\dots,i_{h-1}')=s_{h;l(i_1,\dots,i_{h-1})}(\sigma_a).
\end{equation*}
By noting that for $i=0,1$ there holds $(-1)^i=1-2i $, one can deduce that
\begin{equation*}
\beta(i_1,\dots,i_{h-1})=h-2l(i_1,\dots,i_{h-1}).
\end{equation*}
Therefore, we obtain
\begin{equation}
	B_h^*(\sigma_a)=\sum_{l=1}^{h-1}\binom{h-1}{l}s_{h;l}(\sigma_a)\cos\frac{\pi(h-2l)}{4}=B_h(\sigma_a).
\end{equation}
Now, we consider the contribution of $S_2(x)$. By Lemma \ref{lemma2.1}, we get
\begin{equation}
\int_{T}^{2T}S_2(x)\mathrm{d}x\ll T^{(2+h+2ha)/4}\sum_{(i_1,\dots,i_{h-1})\in \{0,1\}^{h-1}}
\sum_{\substack{n_j\leq y,\,\,1\leq j\leq h \\ \alpha\neq 0}}
\frac{\sigma_a(n_1)\cdots\sigma_a(n_h)}{(n_1 \cdots n_h)^{3/4+a/2}|\alpha|}.
\end{equation}
Hence, it suffices to estimate the sum
\begin{equation*}
{\scriptstyle{\sum}}(y;i_1,\dots,i_{h-1})=\sum\limits_{\substack{n_j\leq y,\,\,1\leq j\leq h\\\alpha\neq 0}}
\frac{\sigma_a(n_1)\cdots\sigma_a(n_h)}{(n_1\cdots n_h)^{3/4+a/2}|\alpha|}
\end{equation*}
for fixed $(i_1,\dots,i_{h-1})\in\{0,1\}^{h-1}$. If $(i_1,\dots,i_{h-1})=(0,\dots,0)$, then
\begin{align*}
{\scriptstyle\sum}(y;0,\dots,0)
&\ll\sum_{n_j\leq y,\,\,1\leq j \leq h}\frac{\sigma_a(n_1)\cdots\sigma_a(n_h)}{(n_1\cdots n_h)^{3/4+a/2}(\sqrt{n_1}+\cdots+\sqrt{n_h})}\\
&\ll\sum_{n_j\leq y,\,\,1\leq j \leq h}\frac{\sigma_a(n_1)\cdots\sigma_a(n_h)}{(n_1\cdots n_h)^{3/4+a/2+1/2h}}\\
&\ll y^{(h-2)/4-ah/2},
\end{align*}
where we used the estimates
\begin{equation*}
\sum_{n\leq u}\sigma_a(n)\ll u,\qquad x_1+\cdots+x_h\gg(x_1\cdots x_h)^{1/h}.
\end{equation*}
For $(i_1,\dots,i_{h-1})\neq (0,\dots,0)$, by a splitting argument, we deduce that there exists a collection of numbers $1<N_1,\dots,N_h<y$ such that
\begin{equation*}
{\scriptstyle\sum}(y;i_1,\dots,i_{h-1})\ll {\scriptstyle\sum ^{*}}\log^hy,
\end{equation*}
where
\begin{equation*}{\scriptstyle\sum^{*}}=\sum\limits_{\substack{N_j<n_j\leq 2N_j,\,\,1\leq j\leq h\\\alpha\neq 0}}
\frac{\sigma_a(n_1)\cdots\sigma_a(n_h)}{(n_1\cdots n_h)^{3/4+a/2}|\alpha|}	.
\end{equation*}
Without loss of generality, we postulate that $N_1\leq \cdots \leq N_h\leq y$, which combined with
lemma \ref{lemma2.3} yields $N_h^{-(2^{h-2}-2^{-1})}\ll |\alpha|\ll y^{1/2}$. Then, by a splitting argument and lemma \ref{lemma2.4},  for some $N_h^{-(2^{h-2}-2^{-1})}\ll|\Delta|\ll y^{1/2}$, we get
\begin{align*}
		   {\scriptstyle\sum^{*}}
\ll & \,\, \frac{y^\varepsilon}{\left(N_1\cdots N_h\right)^{3/4+a/2}\Delta}
           \mathcal{A}\left(N_1,\dots,N_h;i_1,\dots,i_h;\Delta\right)
                   \nonumber \\
\ll & \,\, \frac{y^\varepsilon}{(N_1\cdots N_h)^{3/4+a/2}\Delta}
           \left(\Delta N_h^{1/2}N_1\cdots N_{h-1}+N_1\cdots N_{h-1}\right)
                   \nonumber \\
\ll	& \,\, y^\varepsilon\left(\frac{(N_1\cdots N_{h-1})^{1/4-a/2}}{N_h^{1/4+a/2}}
           +\frac{(N_1\cdots N_{h-1})^{1/4-a/2}}{N_h^{3/4+a/2}\Delta}\right)
                   \nonumber \\
\ll	& \,\, y^\varepsilon\left(N_h^{(h-2)/4-ah/2}+N_h^{b_a(h)}\right)\ll y^{b_a(h)+\varepsilon},
\end{align*}
where $b_a(h)$ is defined as in (\ref{1.10}). Thus, we obtain
	\begin{equation}\label{3.8}
		\int_{T}^{2T}S_2(x)\mathrm{d}x\ll T^{(2+h+2ha)/4+\varepsilon}y^{b_a(h)}.
	\end{equation}
	Henceforth, from (\ref{3.1})--(\ref{3.8}), we have the following lemma.
	\begin{lemma}\label{lemma3.1}
		For any fixed $h\geq 3$, we have
\begin{align}		
\int_{T}^{2T}\mathcal{R}_{a1}^h\mathrm{d}x=\frac{B_h(\sigma_a)}{(\sqrt{2}\pi)^h2^{h-1}}\!
\int_{T}^{2T}x^{h/4+ha/2}\mathrm{d}x
			+O(T^{(4+h+2ha)/4+\varepsilon}y^{-1/2-a}+T^{(2+h+2ha)/4+\varepsilon}y^{b_a(h)}).
\end{align}
\end{lemma}
\noindent
Now we evaluate  the integral $\int_{T}^{2T}\mathcal{R}_{a2}^A\mathrm{d}x$. We begin with the truncated Vorono\"{\i}'s formula
\begin{equation}
\Delta_{a}(x)=\frac{x^{1/4+a/2}}{\sqrt{2}\pi}\sum_{n\leq N}
\frac{\sigma_a(n)}{n^{3/4+a/2}}\cos\left(4\pi\sqrt{nx}-\frac{\pi}{4}\right)+O(x^{1/2+\varepsilon}N^{-1/2}),
\end{equation}
where $1 <N \ll x$. Then we have
\begin{equation}
\mathcal{R}_{a2}(x)=\frac{x^{1/4+a/2}}{\sqrt{2}\pi}\sum_{y<n\leq N}
\frac{\sigma_a(n)}{n^{3/4+a/2}}\cos\left(4\pi\sqrt{nx}-\frac{\pi}{4}\right)+O(x^{1/2+\varepsilon}N^{-1/2}),
\end{equation}
where $1 <N \ll x$. Taking $N=T$, we derive that
\begin{align*}
\mathcal{R}_{a2}=& \,\,\frac{x^{1/4+a/2}}{\sqrt{2}\pi}\sum_{y<n\leq T}
\frac{\sigma_a(n)}{n^{3/4+a/2}}\cos\left(4\pi\sqrt{nx}-\frac{\pi}{4}\right)+O(T^\varepsilon)\notag\\
\ll& \,\, x^{1/4+a/2}\left|\sum_{y<n\leq T}\frac{\sigma_a(n)}{x^{3/4+a/2}}e(2\sqrt{nx})\right|+T^\varepsilon	,
\end{align*}
which implies that
\begin{align}\label{3.12}
\int_T^{2T}\mathcal{R}_{a2}^2\mathrm{d}x
\ll& \,\, T^{1+\varepsilon}+\int_{T}^{2T}\left|x^{1/4+a/2}\sum_{y<n\leq T}
\frac{\sigma_a(n)}{n^{3/4+a/2}}e(2\sqrt{nx})\right|^2\mathrm{d}x\nonumber\\
\ll& \,\, T^{1+\varepsilon}+T^{1/2+a}\int_{T}^{2T}\sum_{y<m,n\leq T}\frac{\sigma_a(m)\sigma_a(n)}{(mn)^{3/4+a/2}}e(2\sqrt{x}(\sqrt{m}-\sqrt{n}))\mathrm{d}x\nonumber\\
\ll& \,\, T^{1+\varepsilon}+T^{3/2+a}\sum_{y<n\leq T}\frac{\sigma^2_a(n)}{n^{3/2+a}}+T^{1+a}\sum_{y<m\neq n\leq T}\frac{\sigma_a(m)\sigma_a(n)}{(mn)^{3/4+a/2}|\sqrt{m}-\sqrt{n}|}\nonumber\\
\ll& \,\, T^{3/2+a}y^{-1/2-a},
\end{align}
where we used the estimates $\sum_{n\leq u}\sigma_a^2(n)\ll u$ and
\begin{align*}
&\sum_{y<m\neq n\leq T}\frac{\sigma_a(m)\sigma_a(n)}{(mn)^{3/4+a/2}|\sqrt{m}-\sqrt{n}|}\\
&=\sum_{\substack{y<m\neq n\leq T\\|\sqrt{m}-\sqrt{n}|\geq (mn)^{1/4}/100}}\frac{\sigma_a(m)\sigma_a(n)}{(mn)^{3/4+a/2}|\sqrt{m}-\sqrt{n}|}+\sum_{\substack{y<m\neq n\leq T\\|\sqrt{m}-\sqrt{n}|< (mn)^{1/4}/100}}\frac{\sigma_a(m)\sigma_a(n)}{(mn)^{3/4+a/2}|\sqrt{m}-\sqrt{n}|}\\
&\ll \sum_{y<m\neq n\leq T}\frac{\sigma_a(m)\sigma_a(n)}{(mn)^{1+a/2}}+\sum_{\substack{y<m\neq n\leq T\\m\asymp n}}\frac{\sigma_a(m)\sigma_a(n)}{(mn)^{1/2+a/2}|m-n|}\\
&\ll\left(\sum_{y<m\leq T}\frac{\sigma_a(m)}{m^{1+a/2}}\right)^2+\sum_{\substack{y<m\neq n\leq T\\m\asymp n}}\frac{\sigma_a^2(m)}{m^{1+a}|m-n|}\ll T^{-a}.
\end{align*}
By using Ivi\'c's large--value technique directly to $\mathcal{R}_{a2}$ without modifications, we derive that
\begin{equation}\label{3.13}
\int_{T}^{2T}\left|\mathcal{R}_{a2}\right|^{A_0}\mathrm{d}x \ll T^{(4+A_0+2A_0a)/4+\varepsilon}.
\end{equation}
For any $2<A<A_0$, combining H\"older's  inequality, (\ref{3.12}) and (\ref{3.13}), we obtain
\begin{align*}	
         \int_{T}^{2T}\left|\mathcal{R}_{a2}\right|^A\mathrm{d}x
= & \,\, \int_{T}^{2T}\left|\mathcal{R}_{a2}\right|^{\frac{2(A_0-A)}{A_0-2}+\frac{A_0(A-2)}{A_0-2}}\mathrm{d}x
               \nonumber \\
\ll & \,\, \left(\int_{T}^{2T}\mathcal{R}_{a2}^2\mathrm{d}x\right)^{\frac{A_0-A}{A_0-2}}
           \left(\int_{T}^{2T}\left|\mathcal{R}_{a2}\right|^{A_0}\mathrm{d}x\right)^{\frac{A-2}{A_0-2}}
               \nonumber \\
\ll	& \,\, T^{\frac{4+A+2Aa}{4}+\varepsilon}y^{-\frac{A_0-A}
	{2(A_0-2)}-\frac{a(A_0-A)}{A_0-A}}.
\end{align*}
Moreover, from(\ref{1.6}) and (\ref{3.13}) we get
\begin{equation}	
\int_{T}^{2T}\left|\mathcal{R}_{a1}\right|^{A_0}dx\ll
\int_{T}^{2T}\left(|\Delta_{a}(x)|^{A_0}+|\mathcal{R}_{a2}|^{A_0}\right)dx\ll T^{(4+A_0+2A_0a)/4+\varepsilon}.
\end{equation}
Thus, we derive the following lemma.
\begin{lemma}
Suppose that $T^{\varepsilon}< y \leq T$, $2<A<A_0$. Then one has
\begin{equation}
\int_{T}^{2T}|\mathcal{R}_{a2}|^{A}dx\ll  T^{\frac{4+A+2Aa}{4}+\varepsilon}y^{-\frac{A_0-A}
	{2(A_0-2)}-\frac{a(A_0-A)}{A_0-A}}.
\end{equation}
\end{lemma}
Now we devote to proving Theorem \ref{maintheorem}. Suppose that $3\leq k\leq A_0$ and $T^\varepsilon < y\leq T$. By the elementary formula $(a+b)^k-a^k\ll|a^{k-1}b|+|b|^k$, we deduce that
\begin{equation}\label{3.17}	\int_{T}^{2T}\Delta_{a}^k(x)\mathrm{d}x=\int_{T}^{2T}\mathcal{R}_{a1}^k\mathrm{d}x
+O\left(\int_{T}^{2T}|\mathcal{R}_{a1}^{k-1}
\mathcal{R}_{a2}|dx\right)+O\left(\int_{T}^{2T}|\mathcal{R}_{a2}|^k\mathrm{d}x\right).
\end{equation}
If $k-1<A_0/2$, then
\begin{equation*}	
\int_{T}^{2T}|\mathcal{R}_{a1}^{k-1}\mathcal{R}_{a2}|\mathrm{d}x\ll
\left(\int_{T}^{2T}|\mathcal{R}_{a1}|^{2(k-1)}\mathrm{d}x\right)^{1/2}\left(\int_{T}^{2T}|
\mathcal{R}_{a2}|^2\mathrm{d}x\right)^{1/2}\ll T^{(4+k+2ka)/4+\varepsilon}y^{-1/4-a/2}.
\end{equation*}
If $k-1\geq A_0/2$, then
\begin{align*}
		   \int_{T}^{2T}|\mathcal{R}_{a1}^{k-1}\mathcal{R}_{a2}|\mathrm{d}x
\ll & \,\, \left(\int_{T}^{2T}|\mathcal{R}_{a1}|^{A_0}\mathrm{d}x\right)^{(k-1)/A_0}
           \left(\int_{T}^{2T}|\mathcal{R}_{a2}|^{A_0/(A_0-k+1)}\mathrm{d}x\right)^{(A_0-k+1)/A_0}
                 \nonumber \\
\ll & \,\, T^{\frac{4+k+2ka}{4}+\varepsilon}y^{-\frac{A_0-A}
	{2(A_0-2)}-\frac{a(A_0-A)}{A_0-A}}.
\end{align*}
Thus, we have
\begin{equation}\label{3.18}	
\int_{T}^{2T}\left|\mathcal{R}_{a1}^{k-1}\mathcal{R}_{a2}\right|
\mathrm{d}x+\int_{T}^{2T}|\mathcal{R}_{a2}|^k\mathrm{d}x\ll T^{(4+k+2ka)/4+\varepsilon}y^{-\alpha(k,A_0)},
\end{equation}
where $\alpha(k,A_0)$ is defined as in (\ref{1.11}). Combining (\ref{3.17}) and (\ref{3.18}), we obtain
\begin{equation}\label{3.19}	
\int_{T}^{2T}\Delta_{a}^k(x)\mathrm{d}x=\int_{T}^{2T}\mathcal{R}_{a1}^k
\mathrm{d}x+O\left(T^{(4+k+2ka)/4+\varepsilon}y^{-\alpha(k,A_0)}\right).
\end{equation}
Taking $y=T^{1/(2b_a(k)+2\alpha(k,A_0))}$. From Lemma \ref{lemma3.1} and (\ref{3.19}), we deduce that
\begin{align}
		 \int_{T}^{2T}\Delta_{a}^k(x)\mathrm{d}x
= & \,\, \int_{T}^{2T}\mathcal{R}_{a1}^k\mathrm{d}x
         +O\left(T^{(4+k+2ka)/4+\varepsilon}y^{-\alpha(k,A_0)}\right)
                \nonumber\\
= & \,\, \frac{B_k(\sigma_a)}{(\sqrt{2}\pi)^k2^{k-1}}
         \int_{T}^{2T}x^{k/4+ka/2}\mathrm{d}x+O\left(T^{(2+k+2ka)/4+\varepsilon}y^{b_a(k)}
         +T^{(4+k+2ka)/4+\varepsilon}y^{-\alpha(k,A_0)}\right)
                \nonumber\\
= & \,\, C_k(a)T^{(4+k+2ka)/4}+O\left(T^{(4+k+2ka)/4-\delta_a(k,A_0)+\varepsilon}\right),
	\end{align}
which completes the proof of Theorem \ref{maintheorem}.
	
\section*{Acknowledgements}
The authors would like to appreciate the referee for his/her patience in refereeing this paper. This work is supported by Beijing Natural Science Foundation (Grant No. 1242003), and the National Natural Science Foundation of China (Grant Nos. 12301006, 12471009, 12071238, 11901566, 12001047, 11971476). Especially, Fei Xue is supported by the China Scholarship Council (Grant No. 202306430075).

\end{document}